\documentclass[12pt]{article}

\date{January 13, 2026}

\PassOptionsToPackage{dvipsnames,svgnames,x11names}{xcolor}
\usepackage{xcolor}
\definecolor{labelkey}{rgb}{0,0.08,0.45}
\definecolor{refkey}{rgb}{0,0.6,0.0}
\definecolor{Brown}{rgb}{0.45,0.0,0.05}
\definecolor{lime}{rgb}{0.00,0.8,0.0}
\definecolor{lblue}{rgb}{0.5,0.5,0.99}
\definecolor{OliveGreen}{rgb}{0,0.6,0}
\definecolor{tyrianpurple}{rgb}{0.4, 0.01, 0.24}
\definecolor{myseagreen}{HTML}{3FBC9D}
\definecolor{myblue}{rgb}{0.9,0.9,0.98}
\colorlet{hlcyan}{cyan!30}

\usepackage{mathpazo} 
\usepackage[T1]{fontenc}
\usepackage{mathtools}
\usepackage{amssymb}
\usepackage{stmaryrd}
\usepackage{empheq}

\usepackage{hyperref}
\hypersetup{
  colorlinks=true,
  linkcolor=refkey,
  urlcolor=lblue,
  citecolor=magenta
}

\usepackage{amsthm}
\makeatletter
\def\th@plain{%
  \thm@notefont{}%
  \itshape 
}
\def\th@definition{%
  \thm@notefont{}%
  \normalfont 
}

\usepackage[capitalize,nameinlink]{cleveref}
\makeatother
\newtheorem{theorem}{Theorem}[section]

\newtheorem{corollary}[theorem]{Corollary}

\newtheorem{example}[theorem]{Example}
\newtheorem{fact}[theorem]{Fact}

\crefname{theorem}{Theorem}{Theorems}
\Crefname{theorem}{Theorem}{Theorems}
\crefname{fact}{Fact}{facts}
\Crefname{fact}{Fact}{facts}
\crefname{equation}{}{equations}
\crefname{chapter}{Appendix}{chapters}
\crefname{item}{}{items}
\crefname{enumi}{}{}

\usepackage{nicematrix}
\usepackage{booktabs}
\usepackage{siunitx}
\usepackage{tikz}
\usepackage[color=myseagreen]{todonotes}
\usepackage{soul}
\usepackage{nameref}
\usepackage{graphicx}
\usepackage{float}
\usepackage[shortlabels,inline]{enumitem}
\setlist[enumerate]{nosep}
\usepackage{relsize}

\usepackage[margin=1in,footskip=0.25in]{geometry}
\makeatletter

\let\orig@label\label

\renewcommand{\label}[1]{%
  \begingroup
  \def\@currentlabelname{}%
  \ifx\current@theorem\relax\else
    \def\@currentlabelname{\current@theorem}%
  \fi
  \ifx\cref@currentlabel\undefined\else
    \let\@currentlabelname\cref@currentlabel
  \fi
  \orig@label{#1}%
  \endgroup
}

\AtBeginEnvironment{theorem}{\def\current@theorem{theorem}}
\AtBeginEnvironment{proposition}{\def\current@theorem{proposition}}
\AtBeginEnvironment{fact}{\def\current@theorem{fact}}
\makeatother

\newcommand{\seppthree}{\setlength{\itemsep}{-3pt}}

\newcommand{\nnn}{\ensuremath{{n\in\mathbb{N}}}}

\newcommand{\menge}[2]{\big\{{#1}~\big|~{#2}\big\}}

\newcommand{\sign}{\ensuremath{\operatorname{sign}}}

\newcommand{\scal}[2]{\left\langle{#1},{#2}\right\rangle}

\newcommand{\RR}{\ensuremath{\mathbb{R}}}

\newcommand{\reli}{\ensuremath{\operatorname{ri}}}

\newcommand{\aff}{\ensuremath{\operatorname{aff}}}

\newcommand{\Id}{\ensuremath{\operatorname{Id}}}

\providecommand{\norm}[1]{\lVert#1\rVert}

\providecommand{\epi}{\operatorname{epi}}

\author{
  Heinz H.\ Bauschke\thanks{
    Mathematics, University of British Columbia,
    Kelowna, B.C.\ V1V~1V7, Canada. E-mail: \texttt{heinz.bauschke@ubc.ca}.}
  ~~~and~~~
  Tran Thanh Tung\thanks{
    Mathematics, University of British Columbia,
    Kelowna, B.C.\ V1V~1V7, Canada. E-mail: \texttt{tung.tran@ubc.ca}.}
}

\title{\textsf{On a result by Meshulam}}

\begin{document}

\maketitle

\begin{abstract}
In 1996, Meshulam proved that every sequence generated 
by applying projections onto affine subspaces, 
drawn from a finite collection in Euclidean space, 
must be bounded. 

In this paper, we extend his result not only from affine subspaces 
to convex polyhedral subsets, but also from Euclidean to general 
Hilbert space. Various examples are provided to illustrate the
sharpness of the results. 
\end{abstract}



\section{Introduction}

Throughout this paper,
\begin{equation}
  \text{$X$ is a real Hilbert space, with inner product $\scal{\cdot}{\cdot}$ and induced norm $\norm{\cdot}$}
\end{equation}
and
\begin{equation}
  \text{$\mathcal{A}$ is a nonempty finite collection of closed affine\footnotemark subspaces of $X$.}
\end{equation}
\footnotetext{Recall that a subset $A$ of $X$ is affine if $A-A$ is a linear subspace.}
Given a nonempty finite collection 
$\mathcal{C}$ of closed convex subsets of $X$ and an interval $\Lambda \subseteq [0,2]$, consider the associated set of relaxed projectors\footnotemark
\begin{equation}
  \mathcal{R}_{\mathcal{C},\Lambda} := \menge{(1-\lambda)\Id+\lambda P_C}{C\in\mathcal{C},\ \lambda\in \Lambda}
\end{equation}
where $P_C$ is the orthogonal projector onto $C$ and $\Id$ is the identity mapping on $X$.
\footnotetext{Given a nonempty closed convex subset $C$ of $X$, we denote by $P_C$ the operator which maps 
$x\in X$ to its unique nearest point in $C$.}

Building on work by Aharoni, Duchet, and Wajnryb \cite{ADW}, Meshulam states in \cite[Theorem~2]{Meshulam} the following result:

\begin{fact}[Meshulam]
\label{f:Meshulam}
Suppose that $X$ is finite-dimensional.
Let $\lambda\in[0,2[$ and $x_0\in X$.
Generate a sequence $(x_n)_{\nnn}$ in $X$ as follows:
Given a current term $x_n$, pick $R_n \in \mathcal{R}_{\mathcal{A},[0,\lambda]}$, and update via
\begin{equation}
  x_{n+1} := R_n x_n.
\end{equation}
Then the sequence $(x_n)_{\nnn}$ is bounded.
\end{fact}

This boundedness result is easy to prove if 
$\bigcap_{A\in \mathcal{A}} A\neq \varnothing$ 
because each $P_A$ is (firmly) nonexpansive and so 
the sequence $(x_n)_\nnn$ is 
Fej\'{e}r monotone. 
The proof of \cref{f:Meshulam} in the case when the intersection 
of the affine subspaces is empty is more involved
and rests on an ingenious induction on the dimension of the space $X$.

In this paper, we provide the following new results on \cref{f:Meshulam}:
\begin{itemize}[itemsep=0pt]
  \item 
We extend \cref{f:Meshulam} from affine subspaces to polyhedral sets 
(see \cref{t:Meshpoly}).
  \item We then extend this to the case when $X$ is 
  infinite-dimensional (see \cref{t:Meshinfi}). 
  This is the main result of this paper. 
  \item We present limiting counterexamples (in \cref{s:counterex} and 
  \cref{s:reflections}). 
\end{itemize}

The rest of the paper is organized as follows: 
After discussing faces, we present the extension of \cref{f:Meshulam} to polyhedral sets in \cref{s:polyhedra}.
In \cref{s:infdim}, we obtain our main result, the extension 
to general Hilbert spaces. 
The sharpness of the results is illustrated 
by two examples provided in \cref{s:counterex}. 
In \cref{s:reflections}, we discuss the situation when 
the relaxation parameters approach $2$.

The notation is standard and follows, e.g., \cite{BC2017} 
and \cite{Rocky}. 

\section{Extension to polyhedral sets}

\label{s:polyhedra}

In this section, we assume that 
\begin{equation} 
\text{
$X$ is finite-dimensional.
}
\end{equation}

\subsection*{Faces and projections of convex sets}

Let $C$ be a nonempty convex subset of $X$.
Recall that a subset $F$ of $C$ is a \emph{face} of $C$ 
if whenever $x,y$ belong to $C$ and 
$F \cap \left]x,y\right[\neq \varnothing$, 
then $x,y$ both belong to $F$. 
Note that $C$ is a face of itself and 
denote the collection of all faces of $C$ by $\mathcal{F}(C)$.
We now set 
\begin{equation}
\label{e:F_c}
(\forall c\in C)\quad 
F_c := 
\bigcap_{c \in F \in \mathcal{F}(C)} F. 
\end{equation}
The following fact is well-known; see, e.g., 
the books by Rockafellar \cite{Rocky} and Webster \cite{Webster} 
(which also contain further information on faces):

\begin{fact}
\label{f:uniqueface}
Let $C$ be a nonempty convex subset of $X$. 
For each $c\in C$, the set $F_c$ (defined in \cref{e:F_c}) is a face of $C$; 
in fact, $F_c$ is the unique face of $C$ with $c\in\reli F_c$. 
Furthermore, 
\begin{equation}
C = \bigcup_{F\in \mathcal{F}(C)} \reli F
\end{equation}
forms a partition of $C$. 
\end{fact}
\begin{proof}
The fact that $F_c$ is a face follows from 
\cite[Theorem~2.6.5(i)]{Webster} while the 
uniqueness property was stated in \cite[Theorem~2.6.10]{Webster}. 
Because two faces whose relative interiors make a 
nonempty intersection must be equal (see \cite[Corollary~18.1.2]{Rocky}), 
we note that the partition claim follows from 
\cite[Theorem~18.2]{Rocky} or \cite[Theorem~2.6.10]{Webster}. 
\end{proof}

Recall that 
the \emph{affine hull} of a nonempty subset $S$ of $X$ is 
the smallest affine subspace of $X$ containing $S$; 
it is denoted by $\aff S$.
We are now in a position to state a recent result 
by Fodor and Pintea (see \cite[Theorem~3.1]{Fodor}):

\begin{fact}
\label{f:Fodor}
Let $C$ be a nonempty closed convex subset of $X$, 
let $F$ be a face of $C$, 
and let $x\in P_C^{-1}(\reli F)$.
Then
\begin{equation}
  P_Cx = P_{\aff F}x. 
\end{equation}
\end{fact}

\begin{corollary}
\label{c:projandface}
Let $C$ be a nonempty closed convex subset of $X$, 
and let $x\in X$. 
Then there exists a face $F$ of $C$ such that 
\begin{equation}
  P_Cx = P_{\aff F}x.
\end{equation}
\end{corollary}
\begin{proof}
Set $c := P_Cx$. 
Then $c\in C$ and obviously $x\in P_C^{-1}(c)$.
In view of \cref{f:uniqueface}, 
the face $F_c$ defined in \cref{e:F_c} 
satisfies $c\in\reli F_c$. 
Set $F := F_c$. 
Altogether, we deduce that $x\in P_C^{-1}(\reli F)$,
and \cref{f:Fodor} implies $P_Cx = P_{\aff F}x$.
\end{proof}

\subsection*{Polyhedral sets and the extension}

Recall that a subset $C$ of $X$ is a \emph{polyhedral set}
or a \emph{polyhedron} if it is the intersection of 
finitely many closed halfspaces.

Polyhedral sets are precisely those closed convex sets 
with finitely many faces (see \cite[Theorem~19.1]{Rocky} or 
\cite[Theorems 3.2.2 and 3.2.3]{Webster}): 

\begin{fact}
\label{f:polychar}
Let $C$ be a nonempty closed convex subset of $X$. 
Then $C$ is a polyhedron if and only if it has finitely many faces.
\end{fact}

We are now ready to extend \cref{f:Meshulam} from 
affine subspaces to polyhedral sets:

\begin{theorem}
\label{t:Meshpoly} 
Let $\mathcal{C}$ be a nonempty finite collection of 
nonempty polyhedral subsets of $X$. 
Let $\lambda\in[0,2[$ and $x_0\in X$.
Generate a sequence $(x_n)_{\nnn}$ in $X$ as follows:
Given a current term $x_n$, pick $R_n \in \mathcal{R}_{\mathcal{C},[0,\lambda]}$, and update via
\begin{equation}
  x_{n+1} := R_n x_n.
\end{equation}
Then the sequence $(x_n)_{\nnn}$ is bounded.
\end{theorem}
\begin{proof}
By \cref{f:polychar},
each polyhedral set $C\in\mathcal{C}$ has finitely many faces, i.e., 
$\mathcal{F}(C)$ is finite. 
Because $\mathcal{C}$ is finite, so is 
$\mathcal{F} := \bigcup_{C\in\mathcal{C}} \mathcal{F}(C)$
and also 
$\{\aff F\}_{F\in\mathcal{F}}$. 
Now suppose that $\mathcal{A} = \{\aff F\}_{F\in\mathcal{F}}$.
If $x\in X$ and $C\in\mathcal{C}$, then 
\cref{c:projandface} 
yields $P_Cx = P_Ax$, for some $A\in\mathcal{A}$ depending 
on $C$ and $x$. 
It follows that the sequence $(x_n)_\nnn$ can be viewed 
as being generated by $\mathcal{A}$. 
Therefore, the result follows from \cref{f:Meshulam}.
\end{proof}

\section{Extension to infinite-dimensional spaces}
\label{s:infdim}

In this section, we assume that 
\begin{equation} 
\text{
$X$ is finite or infinite-dimensional. 
}
\end{equation}

The following representation result for a convex polyhedral set and 
its projection is useful (see \cite[Theorem~3.3.16]{HBthesis}): 

\begin{fact}
\label{f:thesis}
Let $C$ be a nonempty polyhedral subset of $X$, represented 
as 
\begin{equation}
C = \menge{x\in X}{(\forall i\in I)\;\;\scal{a_i}{x}\leq \beta_i}, 
\end{equation}
where $I$ is a finite index set, each $a_i$ belongs to 
$X\smallsetminus\{0\}$ and each $\beta_i$ is a real number.
Let $K$ be closed linear subspace of $\bigcap_{i\in I}\ker a_i$, 
and set $D := C\cap K^\perp$. 
Then $D$ is a nonempty polyhedral subset of $K^\perp$ and 
\begin{equation}
\label{e:thesis}
P_C = P_K + P_DP_{K^\perp}. 
\end{equation}
Note that if $K = \bigcap_{i\in I}\ker a_i$, 
then $K^\perp$ is finite-dimensional. 
\end{fact}

We are now ready to extend \cref{t:Meshpoly} to
general Hilbert spaces: 

\begin{theorem}[main result]
\label{t:Meshinfi} 
Let $\mathcal{C}$ be a nonempty finite collection of 
nonempty polyhedral subsets of $X$. 
Let $\lambda\in[0,2[$ and $x_0\in X$.
Generate a sequence $(x_n)_{\nnn}$ in $X$ as follows:
Given a current term $x_n$, pick $R_n \in \mathcal{R}_{\mathcal{C},[0,\lambda]}$, and update via
\begin{equation}
  x_{n+1} := R_n x_n.
\end{equation}
Then the sequence $(x_n)_{\nnn}$ is bounded.
\end{theorem}
\begin{proof}
For each $C\in\mathcal{C}$, we pick
a finite-codimensional closed 
linear subspace $K_C$ of $X$ as in 
\cref{f:thesis}. 
Next, we set $K := \bigcap_{C\in\mathcal{C}} K_C$.
Then $K$ is a closed linear subspace of $X$, still 
finite-codimensional. 
By \cref{f:thesis} again, we have 
\begin{equation}
\label{e:boom}
(\forall C\in\mathcal{C})\quad
P_C = P_K + P_{D_C}P_{K^\perp},
\end{equation}
where $D_C = C \cap K^\perp$ is a nonempty polyhedral subset of 
the finite-dimensional space $K^\perp$.
The starting point $x_0$ is decomposed as 
$x_0 = k_0 + y_0$, where 
$k_0 := P_Kx_0 \in K$ and $y_0 := P_{K^\perp}x_0 \in K^\perp$. 
Observe that \cref{e:boom} implies that 
$(x_n)_\nnn = (k_0)_\nnn +(y_n)_\nnn$, 
where $(y_n)_\nnn$  
is generated using the finite collection of polyhedral subsets 
$\{D_C\}_{C\in\mathcal{C}}$ in the finite-dimensional space $K^\perp$.
By \cref{t:Meshpoly}, $(y_n)_\nnn$ is bounded, and hence so 
is $(x_n)_\nnn$. 
\end{proof}

\section{Two limiting examples}
\label{s:counterex}

\subsection{Impossibility to extend from polyhedral sets 
to convex sets}

The following example shows that we cannot go beyond polyhedral sets, 
even when we work in the Euclidean plane, 
the collection contains only two sets, one of which is a 
line: 

\begin{example}
Suppose that $X=\RR^2$, and set $C_1 := \RR\times\{0\}$ and 
$C_2 := \epi\exp = \menge{(x,y)\in\RR^2}{\exp(x)\leq y}$. 
Note that while $C_1$ is a (closed) linear subspace of $X$ and 
hence polyhedral, 
the set $C_2$ is closed and convex but not a polyhedral set. 
Because $C_1\cap C_2=\varnothing$, 
the ``gap'' between $C_1$ and $C_2$, 
$\inf\|C_1-C_2\|$, is zero but not attained\footnote{That is, 
the infimum is not a minimum.}. 
Let $x_0\in X$ and generate the sequence of alternating projections 
via 
\begin{equation}
x_{2n+1} := P_{C_1}x_{2n} 
\;\;\text{and}\;\;
x_{2n+2} := P_{C_2}x_{2n+1}.
\end{equation}
By \cite[Corollary~4.6]{02}, we have $\|x_n\|\to\infty$. 
Consequently, \cref{t:Meshpoly} cannot be extended from 
polyhedral sets to general closed convex sets. 
\end{example}

\subsection{Impossibility to extend Meshulam's result to general  
Hilbert spaces} 

While we succeeded in generalizing Meshulam's result 
to polyhedral sets in general Hilbert spaces 
(see \cref{t:Meshinfi}), it is not possible to extend his 
result to general closed affine subspaces in general Hilbert spaces: 

\begin{example}
Following \cite[Example~4.3]{02}, there exists 
a (necessarily infinite-dimensional) Hilbert space $X$ 
with two closed affine subspaces $A_1$ and $A_2$ such that 
their ``gap'' $\inf\|A_1-A_2\|$ is equal to $1$ but not attained. 
Let $x_0\in X$ and generate the sequence of alternating projections 
via 
\begin{equation}
x_{2n+1} := P_{A_1}x_{2n} 
\;\;\text{and}\;\;
x_{2n+2} := P_{A_2}x_{2n+1}.
\end{equation}
By \cite[Corollary~4.6]{02}, we have $\|x_n\|\to\infty$. 
Consequently, \cref{f:Meshulam} cannot be extended from 
Euclidean spaces to general Hilbert spaces.
\end{example}

\section{Approaching reflections} 
\label{s:reflections}

Let us return to a specific setting of \cref{f:Meshulam}:
Suppose that $\mathcal{A}=\{A_1,A_2\}$, where
$A_1$ and $A_2$ are affine subspaces of $X$, assumed to 
be finite-dimensional. 
Pick a sequence of relaxation parameters 
$(\lambda_n)_{\nnn}$ in $[0,2]$, a starting point $x_0\in X$, and
form the sequence of alternating relaxed projections via 
\begin{subequations}
\label{e:0623a}
\begin{align}
x_{2n+1} &:= (1-\lambda_{2n}) x_{2n} + \lambda_{2n} P_{A_1}x_{2n}, \\
x_{2n+2} &:= (1-\lambda_{2n+1}) x_{2n+1} + \lambda_{2n+1} 
P_{A_2}x_{2n+1}.
\end{align}
\end{subequations}
It follows from \cref{f:Meshulam} that 
\begin{equation}
\text{
if $\textstyle \varlimsup_n\lambda_n<2$,
then the sequence $(x_n)_{\nnn}$ is bounded. 
}
\end{equation}

We now investigate the case when $\varlimsup_n\lambda_n=2$; 
in other words, a subsequence of the operators $R_n$ approximate 
reflections. 
We specialize further to the case when 
$X=\RR$, $A_1=\{a\}$, and $A_2=\{b\}$, where  
$a,b$ are distinct\footnote{If $a=b$, then $A_1\cap A_2\neq \varnothing$,
and the sequence $(x_n)_\nnn$ stays bounded no matter how 
$(\lambda_n)_\nnn$ is chosen in $[0,2]$.} real numbers.
First, \cref{e:0623a} turns into 
\begin{subequations}
\label{e:0623b}
\begin{align}
x_{2n+1} &:= (1-\lambda_{2n}) x_{2n} + \lambda_{2n} a, \\
x_{2n+2} &:= (1-\lambda_{2n+1}) x_{2n+1} + \lambda_{2n+1} b.
\end{align}
\end{subequations}
Note that $x_{2n+1}-x_{2n} = \lambda_{2n}(a-x_{2n})$; 
thus if $(x_{2n})_\nnn$ is bounded, then so is $(x_{2n+1})_\nnn$.
Thus, we focus on the even terms, and we set 
for every $k,n\in\mathbb{N}$, 
\begin{subequations}
\begin{align}
y_{n} &:= x_{2n}, \;
\varepsilon_n := 2-\lambda_n, \; \\
\gamma_n &:= (1-\lambda_{2n})(1-\lambda_{2n+1}) =
1 -\varepsilon_{2n} - \varepsilon_{2n+1} + 
\varepsilon_{2n}\varepsilon_{2n+1}, \;\\
\delta_n &:= 1-\gamma_n, \;
d_n := (1-\lambda_{2n+1})\lambda_{2n}a + \lambda_{2n+1} b, \\
\Gamma_{k,n} &:= \begin{cases}
\gamma_k\cdots \gamma_{n} = (1-\delta_k)\cdots(1-\delta_{n}),\text{ if }k\leq n,\\
1,\text{ if }k> n.
\end{cases}
\end{align}
\end{subequations}
Then \cref{e:0623b} leads to the recursion 
\begin{equation}
\label{e:0623c-}
y_{n+1} = \gamma_n y_n + d_n, 
\end{equation}
which has the solution 
\begin{equation}
\label{e:0623c}
y_n = \Gamma_{0,n-1} y_0 + \sum_{k=0}^{n-1} d_k {\Gamma_{k+1,n-1}}.
\end{equation}

\subsection{The case $\lim_n \lambda_n = 2$}

We now assume that $1<\lambda_n\to 2$. 
It follows that $\varepsilon_n\to 0$, 
$\gamma_n\to 1^-$, $\delta_n\to 0^+$, and 
$d_n \to 2(b-a)$. 
Note that since $\varepsilon_n\in[0,1[$ for all $n\in\mathbb{N}$, we obtain
\begin{subequations}
\label{260113}
\begin{align}
\sum_n\varepsilon_n&=\sum_n(\varepsilon_{2n}+\varepsilon_{2n+1})\\
&\geq\sum_n(\varepsilon_{2n}+\varepsilon_{2n+1}-\varepsilon_{2n}\varepsilon_{2n+1})=\sum_n\delta_n\\
&\geq\sum_n\varepsilon_{n}-\frac{1}{2}\sum_n(\varepsilon_{2n}^2+\varepsilon_{2n+1}^2)\\
&=\sum_n\varepsilon_{n}-\frac{1}{2}\sum_n\varepsilon_{n}^2\\
&\geq\frac{1}{2}\sum_n\varepsilon_{n};
\end{align}
\end{subequations}
consequently, $\sum_n\varepsilon_n<+\infty$ if and only if $\sum_n\delta_n<+\infty$.

\subsubsection{The subcase $\sum_n\varepsilon_n<+\infty$}
\label{sss:finite}
We now assume that $\sum_n\varepsilon_n<+\infty$.
It follows that 
$\sum_n\delta_n<+\infty$ and hence 
\begin{equation}
\Gamma_{0,n} = \prod_{k=0}^{n} (1-\delta_k) \to \ell > 0. 
\end{equation}
Then there exists $m\in\mathbb{N}$ such that
\begin{equation}
\left(\forall m\leq k \leq n\right)\quad\sign\left(d_k\Gamma_{k+1,n-1}\right) = \sign(b-a)\text{ and } |d_k| \Gamma_{k+1,n-1}>|b-a|{\ell}. 
\end{equation}
Hence \cref{e:0623c} yields 
$y_n\to \sign(b-a)\infty$. 
Consequently,
\begin{equation}
x_{2n} \to \sign(b-a)\infty
\;\;\text{and}\;\;
x_{2n+1} \to \sign(a-b)\infty.
\end{equation}

\subsubsection{The subcase $\sum_n\varepsilon_n=+\infty$}
We now assume that $\sum_n\varepsilon_n=+\infty$.
It follows that $\sum_n\delta_n=+\infty$ and hence 
\begin{equation}
\Gamma_{0,n} = \prod_{k=0}^{n} (1-\delta_k) \to 0. 
\end{equation}
Let's us focus on a case that allows for closed-form expressions. 
Set 
\begin{equation}
(\forall\nnn)\quad 
\lambda_n=\frac{2n+3}{n+2};
\;\;\text{hence,}\;\;
\gamma_n = \frac{2n+1}{2n+3}\;\text{and}\;
\delta_n = \frac{2}{2n+3}. 
\end{equation}
Hence, using telescoping, we have 
$\Gamma_{k,n} = (2k+1)/(2n+3)$. 
Furthermore,  \cref{e:0623c} yields 
\begin{equation}
y_n = \frac{1}{2n+1}y_0 + \sum_{k=0}^{n-1} d_k \frac{2k+3}{2n+1}
= \frac{1}{2n+1}y_0 + \frac{1}{n}\sum_{k=0}^{n-1} \frac{d_k(2k+3)}{2+1/n}.
\end{equation}
Because $\lim_k d_k(2k+3)=\sign(b-a)\infty$, we deduce that 
$\lim_n y_n=\sign(b-a)\infty$ as well (using Ces\`aro averages). 
Consequently, as before, 
\begin{equation}
x_{2n} \to \sign(b-a)\infty
\;\;\text{and}\;\;
x_{2n+1} \to \sign(a-b)\infty.
\end{equation}
In the general case, we cannot telescope and the analysis is 
likely significantly harder.

\subsection{The case $\varliminf_n\lambda_n<\varlimsup_n \lambda_n = 2$}

No general conclusions can be drawn in this case---%
we shall present two examples that show that the sequence 
$(x_n)_\nnn$ may have quite different behaviours.

\subsubsection{A case where $(x_n)_\nnn$ is bounded} 
\label{sss:bounded}
Assume that 
\begin{equation}
(\forall\nnn)\quad 
\lambda_{2n} = 2-\rho_n\in\left]1,2\right[
\;\;\text{and}\;\;
\lambda_{2n+1} = 1, 
\end{equation}
and $\rho_n\to 0$. 
Then $(\forall\nnn)$ $d_n = b$. 
Moreover, $\gamma_n=0$, and \cref{e:0623c-} yields 
$y_{n+1}=d_n= b$. 
Hence $x_{2n+1}=(\rho_n-1)y_n + (2-\rho_n)a \to 
-b+2a = 2a-b$. 
To sum up, 
\begin{equation}
x_{2n}\to b
\;\;\text{and}\;\;
x_{2n+1}\to 2a-b;
\end{equation}
in particular, the sequence $(x_n)_\nnn$ is bounded.

\subsubsection{A case where $(x_n)_\nnn$ is unbounded} 
\label{sss:mixed}

Now assume that $(\mu_n)_\nnn$ is a relaxation parameter 
sequence such that $|x_n|\to +\infty$ 
(e.g., take $\mu_n = 2-1/2^{n+1}$ and recall 
\cref{sss:finite}). 
Using this, we create another relaxation parameter sequence 
by truncating $(\mu_n)_\nnn$ and resetting to $1$ in the following way:
\begin{equation}
(\lambda_n)_\nnn := 
\big(\mu_0,1,\mu_0,\mu_1,\mu_2,1,\mu_0,\mu_1,\mu_2,\mu_3,\mu_4,1,
\mu_0,\mu_1,\mu_2,\mu_3,\mu_4,\mu_5,\mu_6,1,\ldots\big). 
\end{equation}
Then for every $\nnn$, $\lambda_{n(n+1)-1}=1$ and 
so $x_{n(n+1)} = b$. 
Because the stretches between consecutive ones grow longer and longer, 
we deduce that $|x_{n(n+1)-1}|\to\infty$.  
Thus, $(x_n)_\nnn$ is unbounded  with 
$\varliminf_n|x_n|<+\infty$. 

\subsection{Summary}

To sum up, we've seen that 
the case when $\varlimsup_n\lambda_n=2$ exhibits quite subtle
behaviour:
\begin{itemize}[itemsep=0pt]
\item The sequence $(x_n)_\nnn$ may be convergent (even 
constant, when $x_0=a=b$). 
\item The sequence $(x_n)_\nnn$ may be bounded 
but not convergent 
(see \cref{sss:bounded}). 
\item The sequence $(x_n)_\nnn$ may satisfy 
$\varliminf_n|x_n|<\varlimsup_n|x_n|=+\infty$ 
(see \cref{sss:mixed}). 
\item The sequence $(x_n)_\nnn$ may satisfy 
$\lim_n|x_n|=+\infty$ 
(see \cref{sss:finite}). 
\end{itemize}

\section*{Acknowledgments}
The authors thank the handling editor, Dr.~Walaa Moursi, and the referee 
for constructive and helpful comments. 
The research of HHB was supported by a Discovery Grant from the Natural Sciences and Engineering Research Council of Canada.


\begin{thebibliography}{999}
\seppthree

\bibitem{ADW}
R.\ Aharoni, P.\ Duchet, and B.\ Wajnryb:
Successive projections on hyperplanes,
\emph{Journal of Mathematical Analysis and Applications}~103 (1984), 134--138.
\url{https://doi.org/10.1016/0022-247X(84)90163-X}

\bibitem{HBthesis}
H.H.\ Bauschke: 
\emph{Projection Algorithms and Monotone Operators}, 
PhD thesis, Mathematics, Simon Fraser University, Burnaby, B.C., Canada, 1996. 
\url{https://summit.sfu.ca/item/7015}

\bibitem{02}
H.H.\ Bauschke and J.M.\ Borwein:
Dykstra's alternating projection algorithm for two sets,
\emph{Journal of Approximation Theory}~79 (1994), 418--443.
\url{https://doi.org/10.1006/jath.1994.1136}. 

\bibitem{BC2017}
H.H.\ Bauschke and P.L.\ Combettes:
\emph{Convex Analysis and Monotone Operator Theory in Hilbert Spaces},
2nd edition, Springer, 2017.
\url{https://doi.org/10.1007/978-3-319-48311-5}

\bibitem{Fodor}
V.A.\ Fodor and C.\ Pintea:
The metric projection over a polyhedral set
through the relative interiors of its faces, 
\emph{Optimization Letters} (2025), 
\url{https://doi.org/10.1007/s11590-025-02184-7} 


\bibitem{Meshulam}
R.\ Meshulam:
On products of projections,
\emph{Discrete Mathematics}~154 (1996), 307--310.
\url{https://doi.org/10.1016/0012-365X(95)00055-2}

\bibitem{Rocky}
R.T.\ Rockafellar:
\emph{Convex Analysis},
Princeton University Press, 1970. 

\bibitem{Webster}
R.\ Webster:
\emph{Convexity},
Oxford University Press, 1994.

\end{thebibliography}
\end{document}